\documentclass{ws-ijnt}

\usepackage{xcolor}
\usepackage{amssymb}
\usepackage{amsmath}
\usepackage[square,numbers]{natbib}
\begin{document}

\markboth{Andrew Rajchert}{On the Glasner Property of Linear Maps with Prime Entries on Tori}

%
\catchline{}{}{}{}{}
%

\title{On the Glasner Property of Linear Maps with Prime Entries on Tori}

\author{Andrew Rajchert}

\address{School of Mathematics and Statistics, The University of Sydney, \\
NSW, 2006, Australia\,\\
\email{araj9923@uni.sydney.edu.au}
}


\maketitle


\begin{abstract}
We study the quantitative Glasner property in the context of maps between tori of differing dimension, instead of as a (semi-)group action. We also only consider matrices with entries being non-constant polynomials evaluated at primes, extending the work of Nair and Velani \cite{NV} and Bulinski and Fish \cite{BF} to a more general setting.
\end{abstract}



\section{Introduction}

\begin{definition}
\label{Glasner}
Let $X$ and $Y$ be metric spaces, with $Y$ compact, and let $\mathcal{M}$ be a set of functions mapping $X$ to $Y$. We say the triple $(X, Y, \mathcal{M})$ has the Glasner property if, for any choice of $\epsilon > 0$ and any choice of infinite set $A \subseteq X$, there exists an $f \in \mathcal{M}$ such that $f(A)$ is $\epsilon$-dense in $Y$.
\end{definition}

By $\epsilon$-dense, we mean that for any choice of $y \in Y$, there exists a point $x \in A$ such that $d(y, f(x)) < \epsilon$ with $d$ being the distance metric of $Y$. It is natural to ask what can be attained if our subset $A$ is only finite, which leads to a quantitative variant of the Glasner property which imposes additional constraints.

\begin{definition}
\label{QuantGlasner}
The triple $(X, Y, \mathcal{M})$ has the quantitative Glasner property if it has the Glasner Property and there exists some $k : (0, \infty) \to \mathbb{R}$ such that for any $\epsilon > 0$ and $A \subseteq X$ with $|A| > k(\epsilon)$, there exists an $f \in \mathcal{M}$ such that $f(A)$ is $\epsilon$-dense in $Y$.
\end{definition}

Both definitions \ref{Glasner} and \ref{QuantGlasner} are different to the standard Glasner property and quantitative variants studied in other works, as here we do not impose the restriction that $X = Y$.
We study the quantitative Glasner property in the context of the torus, $\mathbb{T}^n = \mathbb{R}^n/\mathbb{Z}^n$. That is, $X$ and $Y$ are tori (with possibly different dimension) and $\mathcal{M}$ is the set of functions given by matrix multiplication with the matrices coming from some subset of the matrices with integer entries. 
From this point on, $\mathcal{M}$ will only consist of functions of the form $f(x) = Mx$, for some coefficient $M$, possibly a matrix. In this case, it can be easier to define $\mathcal{M}$ directly from the set of coefficients being considered rather than the functions, which we will do. For example, it is known $(\mathbb{T}, \mathbb{T}, \mathbb{Z})$ has the Glasner property.

In the context of the Glasner property, the torus is a well-studied space. Alon and Peres show that for a fixed choice of $\delta > 0$, $(\mathbb{T}, \mathbb{T}, \mathbb{Z})$ has the quantitative Glasner property with $k(\epsilon) = 1/\epsilon^{2+\delta}$ as a lower bound for all positive $\epsilon < \epsilon_0$, where $\epsilon_0$ depends on the choice of $\delta$. Furthermore, they show this choice of $k(\epsilon)$ is optimal up to the choice of $\delta$ (i.e. $\delta$ cannot be chosen as zero). Alon and Peres also extend this to show $(\mathbb{T}, \mathbb{T}, P)$ has the quantitative Glasner property with the same $k(\epsilon)$, where $P$ is the set of primes. Additionally, for any non-constant polynomial $f(x) \in \mathbb{Z}[x]$, Alon and Peres extend their result to $(\mathbb{T}, \mathbb{T}, f(\mathbb{Z}))$, with $k(\epsilon)$ depending on the degree of $f$ \cite{AP}. Their Harmonic Analysis-based approach allows for relatively minor adjustments to prove a variety of different triples have the quantitative Glasner property, which has been used by many to prove other variations of this statement.

For our purposes, one of the more important variations of Alon and Peres' work is that of Nair and Velani, whose theorem we state now \cite{NV}:

\begin{theorem}
\label{theoremVN}
Let $P$ be the set of primes. Fix $\delta > 0$ and some $f \in \mathbb{Q}[x]$ where $f$ is a non-constant polynomial of degree $L \geq 1$ with $f(\mathbb{Z}) \subseteq \mathbb{Z}$. Then there exists some $\epsilon_0 = \epsilon_0(f, \delta)$ such that with $\mathcal{M} = f(P)$, $(\mathbb{T}, \mathbb{T}, \mathcal{M})$ has the quantitative Glasner property where $k(\epsilon) = 1/\epsilon^{2L + \delta}$ for all $\epsilon < \epsilon_0$.
\end{theorem}

Another important variant is that of Bulinski and Fish, who also work with polynomials, but in the higher dimensional case \cite{BF}.

\begin{theorem}
\label{theoremBF}
Let $n \geq 1$ be an integer, $\delta > 0$, and $f_{i, j} \in \mathbb{Z}[x]$ be non-constant polynomials of degree $L_{i, j} \geq 1$ for $1 \leq i, j \leq n$. Set $L = \max_{i, j}(L_{i, j})$. Additionally, assume that the $f_{i, j}$ and 1 are mutually linearly independent over $\mathbb{Z}$ and let $\mathcal{M}$ be the set of matrices, 
\begin{align*}
\mathcal{M} = \left\{\begin{bmatrix}
f_{1, 1}(x) & f_{1, 2}(x) & ... & f_{1, n}(x)\\
f_{2, 1}(x) & f_{2, 2}(x) & ... & f_{2, n}(x)\\
 & \vdots & & \\
f_{n, 1}(x) & f_{n, 1}(x) & ... & f_{n, n}(x)\\
\end{bmatrix}, x \in \mathbb{Z} \right\}.
\end{align*}

Then there exists some $\epsilon_0 = \epsilon_0(\mathcal{M}, \delta) > 0$ such that $(\mathbb{T}^n, \mathbb{T}^n, \mathcal{M})$ has the quantitative Glasner property, where $k(\epsilon) = 1/\epsilon^{n(n+1)(2L+1) + \delta}$ for all $\epsilon < \epsilon_0$.
\end{theorem}

The bound in theorem \ref{theoremBF} was recently improved upon by Shparlinski by improving on Hua's bound for exponential sums in special cases, decreasing the exponent on $\epsilon$ to $n(2n+1)L + (7n+1)/2 + \delta$ \cite{Sh2}. While this is stronger, we do not place any significant weight on the quality of the bound $k(\epsilon)$, and so for the sake of brevity, we do not attempt to merge his work into our results.

For our work, we merge theorems \ref{theoremVN} and \ref{theoremBF} by focusing on polynomials $f_{i, j}$ that are evaluated at the primes in the higher dimensional case, in addition to expanding to the case where our two tori are of different dimensions. This allows us to study the impact of the input space and output spaces separately, and how they may affect the lower bound $k(\epsilon)$. Our first theorem is the following:

\begin{theorem}
\label{theorem1}
Fix $\delta > 0$, positive integers $m, n \geq 1$, and non-constant polynomials $f_{i, j} \in \mathbb{Q}[x]$ for $1 \leq i \leq n, 1 \leq j \leq m$ of degrees $L_{i, j} \geq 1$ with $f_{i, j}(\mathbb{Z}) \subseteq \mathbb{Z}$. Let $L = \max(L_{i, j})$ and let $\mathcal{M}$ be the set of matrices,
\begin{align*}
\mathcal{M} = \left\{\begin{bmatrix}
f_{1, 1}(p_{1,1}) & f_{1, 2}(p_{1,2}) & ... & f_{1, m}(p_{1,m})\\
f_{2, 1}(p_{2,1}) & f_{2, 2}(p_{2,2}) & ... & f_{2, m}(p_{2,m})\\
 & \vdots & & \\
f_{n, 1}(p_{n,1}) & f_{n, 1}(p_{n,2}) & ... & f_{n, m}(p_{n,m})\\
\end{bmatrix}, p_{i,j} \text{ are prime} \right\}.
\end{align*}

Then there exists a $\epsilon_0 = \epsilon_0(\mathcal{M}, \delta) > 0$ such that $(\mathbb{T}^m, \mathbb{T}^n, \mathcal{M})$ has the quantitative Glasner property, where for all $\epsilon \in (0, \epsilon_0)$,
\begin{align*}
k(\epsilon) = 1/\epsilon^{4Lmn+\delta}.
\end{align*}

Furthermore, in the case that $m = 1$, we instead have the stronger bound $k(\epsilon) = 1/\epsilon^{2Ln + \delta}$ for all $\epsilon \in (0, \epsilon_0)$.
\end{theorem}

Theorem \ref{theorem1} can be thought of as the extension of theorem \ref{theoremVN} to the higher dimensional case, whose bound we match when $m = n = 1$ as expected. Additionally, when compared to theorem \ref{theoremBF}, we have dropped the requirement that the $f_{i, j}$ and 1 are mutually linearly independent over $\mathbb{Z}$. We are able to remove this requirement due to each entry of the matrices in $\mathcal{M}$ depending on a single prime $p_{i,j}$ which can be chosen without regard to any other entry. In contrast, every entry of the matrices in $\mathcal{M}$ of theorem \ref{theoremBF} are determined through a single choice of integer, and thus must be considered together.


We choose to study polynomials evaluated at primes due to several known results describing bounds on exponential sums involving sequences of the form $\{f(p_n)\}_{n \geq 1}$ where $f$ is a polynomial \footnote{Namely, the results used are described in lemmas \ref{lemma2} and \ref{lemma3}.}. It is straight forward to show that if similar bounds are known for other sequences, $\{x_{i, j, n}\}_{n \geq 1}$ for any $i, j$, then the entries of the matrices in $\mathcal{M}$ can be replaced by the $x_{i, j, n}$. For the sake of keeping the statement of theorem \ref{theorem1} simple while not sacrificing too much generality, we focus on polynomials evaluated at primes, where the required bounds on exponential sums are already known to be true.

Of interest is the necessity of how our lower bound, $k(\epsilon)$, depends on the parameters $n, m$ and $L$. Indeed, the dependence on $n$ is essential, which can quickly be shown by contradiction. If it were possible to have $k(\epsilon)$ not depend on $n$ while maintaining the quantitative Glasner Property, fix $\epsilon$ small enough and let $n$ and $A \subseteq \mathbb{T}^m$ be such that $k(\epsilon) < |A| < \lfloor 1/\epsilon \rfloor^n$. Then any mapping from $A$ to $\mathbb{T}^n$ cannot be $\epsilon$-dense since the union of balls centred at the mapped points with radii $\epsilon$ cannot cover a volume larger than the volume of $\mathbb{T}^n$. This contradicts the definition of the quantitative Glasner Property, demonstrating the necessity of $n$ in $k(\epsilon)$.

While not initially obvious, the dependence on $m$ is also required. If it were possible to have $k(\epsilon)$ not depend on $m$, then take $m$ such that $2^m > k(1/5)$. Let $A = \{0, 1/2\}^m \subseteq \mathbb{T}^m$ be the set of points with coordinates $0$ or $1/2$, and notice $|A| = 2^m > k(1/5)$. From the definition of $\mathcal{M}$, we know $\mathcal{M} \subseteq M_{n \times m}(\mathbb{Z})$. It follows that for any $M \in \mathcal{M}, MA \subseteq \{0, 1/2\}^n$, since any $\mathbb{Z}$ linear combination of $0$ and $1/2$ is equal to $0$ or $1/2$ modulo 1. Clearly the ball centred at $(1/4, 1/4, ..., 1/4) \in \mathbb{T}^n$ with radius $1/5$ does not intersect $\{0, 1/2\}^n$, so $MA$ cannot be $\epsilon$-dense with $\epsilon = 1/5$ for any choice of $M \in \mathcal{M}$. This goes against the definition of the quantitative Glasner property, and hence $k(\epsilon)$ must depend on $m$.

The dependence on $L$ is not obvious from the outset, and unfortunately it is not known if this is essential. The $L$ term originates from the use of the Hua bound to get an asymptotic upper bound on the value of an exponential sum, which in turn originates from the typical Harmonic Analysis approach to proving statements about the quantitative Glasner property. As the Hua bound is effectively tight, the $L$ dependence cannot be avoided without significant changes to the proof used.

The separate $m = 1$ case in theorem \ref{theorem1} originates from being able to use a tighter bound than usual in our proof. More specifically, when given the set $\{x_i \mid 1 \leq i \leq k\} \subseteq \mathbb{T}^m$ and an integer $b$, a key section of the proof requires us to bound the number of pairs $1 \leq i < j \leq k$ such that $x_i-x_j$ only has rational entries with small denominators. In the case of $m = 1$, we are able to make use of proposition 1.3 of Alon and Peres \cite{AP}, a statement which fails to extend naturally to the higher dimensional case. When $m > 1$ we must instead use a weaker argument based on proposition 1 from Kelly and L\^e \cite{KL}. An obvious path for future work is delving more into extending proposition 1.3 of Alon and Peres into the higher dimensional case to produce a tighter bound when $m > 1$.


We also prove another theorem, which instead considers the case where the elements of $\mathcal{M}$ are defined from a single choice of prime, rather than a separately chosen prime for each entry of the matrix.

\begin{theorem}
\label{theorem2}
Fix $\delta > 0$, positive integers $m, n \geq 1$, and non-constant polynomials $f_{i, j} \in \mathbb{Z}[x]$, $1 \leq i \leq n$, $1 \leq j \leq m$ of degrees $L_{i, j} \geq 1$ such that the $f_{i, j}$ and $1$ are mutually linearly independent over $\mathbb{Z}$. Let $L = \max(L_{i, j})$ and $\mathcal{M}$ be the set of matrices,
\begin{align*}
\mathcal{M} = \left\{\begin{bmatrix}
f_{1, 1}(p) & f_{1, 2}(p) & ... & f_{1, m}(p)\\
f_{2, 1}(p) & f_{2, 2}(p) & ... & f_{2, m}(p)\\
 & \vdots & & \\
f_{n, 1}(p) & f_{n, 1}(p) & ... & f_{n, m}(p)\\
\end{bmatrix}, p \text{ is prime} \right\}.
\end{align*}

Then there exists a $\epsilon_0 = \epsilon_0(\mathcal{M}, \delta) > 0$ such that $(\mathbb{T}^m, \mathbb{T}^n, \mathcal{M})$ has the quantitative Glasner property, where for all $\epsilon \in (0, \epsilon_0)$,
\begin{align*}
    k(\epsilon) = 1/\epsilon^{(2L+1)(m+1)n + \delta}.
\end{align*}

Furthermore, in the case that $m = 1$, we instead have the stronger bound $k(\epsilon) = 1/\epsilon^{(2L+1)n + \delta}$ for all $\epsilon \in (0, \epsilon_0)$.
\end{theorem}

Due to the single choice of prime we index $\mathcal{M}$ over, this can be thought of as an extension of theorem \ref{theoremBF} into the non-square and prime case. Additionally, these extra restrictions do not affect the quality of the lower bound $k(\epsilon)$ found by Bulinski and Fish.

\section{Proof of theorem \ref{theorem1}}

Before we prove the main theorem, there is a variety of lemmas that will be useful. We first recall lemma 6.2 from Alon and Peres \cite{AP} on the existence and behaviour of bump functions.

\begin{lemma}
\label{lemma1}
For $0 < \epsilon < 1$, there exist non-negative functions $g_\epsilon : \mathbb{R} \to \mathbb{R}$ of period 1 such that $g_\epsilon(t) = 0$ for $\epsilon \leq |t| \leq 1/2$ and the Fourier coefficients $\widehat{g_\epsilon}(m)$ satisfy $\widehat{g_\epsilon}(0) = 1$ and,
\begin{align*}
\forall m \in \mathbb{Z}, \quad |\widehat{g_\epsilon}(m)| \leq C\exp(-\sqrt{\epsilon |m|}),
\end{align*}
where $C$ is an absolute constant. Furthermore, there exists another constant $C' > 0$ such that for any $\epsilon \in (0, 1)$,
\begin{align*}
\int_{-\epsilon}^\epsilon g_\epsilon(z)^2dz = \dfrac{C'}{\epsilon}.
\end{align*}
\end{lemma}

We will note that lemma 1 of Kelly and L\^e \cite{KL} gives the stronger result that $\widehat{g_\epsilon}(m) = 0$ for all $m$ sufficiently far enough from zero, but this does not meaningfully improve our results or significantly simplify our working, so we use the result from Alon and Peres instead.

We will also use a lemma on the equidistribution of irrational polynomials from Rhin \cite{R}. For the remainder of this paper, let $P_N$ denote the first $N$ primes.

\begin{lemma}
\label{lemma2}
Let $f$ be a real polynomial such that $f(x) - f(0)$ has at least one irrational coefficient. Then,
\begin{align*}
\lim_{N \to \infty}\dfrac{1}{N}\sum_{p \in P_N}\exp(2\pi i f(p)) = 0.
\end{align*}
\end{lemma}

Additionally, we summarise part of the proof of theorem 1 from Nair and Velani in \cite{NV} in the following lemma,

\begin{lemma}
\label{lemma3}
Let $f: \mathbb{N} \to \mathbb{N}$ be a non-constant polynomial, and $a/b$ be a non-zero rational in reduced form. Then for any choice of $\epsilon' > 0$, there exists a constant $c_2 = c_2(\epsilon', f)$ such that,
\begin{align*}
\sum_{w = 1}^T \left| \lim_{N \to \infty} \dfrac{1}{N} \sum_{p \in P_N} \exp\left(2\pi i wf(p) \dfrac{a}{b}\right)\right| \leq c_2 T b^{-1/\deg(f) + \epsilon'}.
\end{align*}
\end{lemma}

Our final necessary lemma is adapted from proposition 2.7 from Bulinski and Fish \cite{BF}, which is in turn based on proposition 1.3 from Alon and Peres \cite{AP}, which we alter and prove in a version suitable for our purposes. We will first define some helpful notation.

\begin{definition}
Fix $S = \{z_1, ..., z_k\} \in \mathbb{T}^m$, a set of $k$ distinct elements in the torus. Then define $h_b = h_b(S)$ to be the number of pairs $(i, j)$ such that $1 \leq i < j \leq k$ and $b$ is the smallest positive integer such that, $b(z_i-z_j) \in \mathbb{Z}^m$. Additionally, let $H_b = \sum_{t=1}^b h_t$.

Furthermore, define $s_b = s_b(S)$ to be the number of pairs $(i, j)$ such that $1 \leq i < j \leq k$, $z_i - z_j$ only contains rational entries and the maximum denominator of the reduced rational entries is $b$. Similarly, let $S_b = \sum_{t=1}^b s_t$.
\end{definition}

Observe that the definition of $h_b$ and $s_b$ maps each pair of points in $S$ to at most a single integer, and hence $H_b \leq k^2$ and $S_b \leq k^2$ always holds true. Also observe that in the case $m = 1$, we have $s_b = h_b$ for all $b$. While we will usually refer to $s_b$ throughout the proof of theorem \ref{theorem1} and $h_b$ in the proof of theorem \ref{theorem2}, we define both here to make the connection between the two clear.

\begin{lemma}
\label{lemma4}
Fix any set $S = \{x_1, ..., x_k\} \subseteq \mathbb{T}^m$ and $s_b = s_b(S)$. Then for any choice of $T > 0$, $c > 0$ and $r \in (0, 1)$, there exists a constant $C$ independent of $T$ and $k$ such that, 
\begin{align*}
\sum_{b = 2}^\infty s_b \left(\left( cTb^{-r}+1\right)^n-1\right) \leq CT^nk^{2-r/(2m)}.
\end{align*}

Furthermore, in the case of $m = 1$, for any $\alpha > 0$, there exists a $k_0$ large enough such that if $k > k_0$,
\begin{align*}
\sum_{b = 2}^\infty s_b \left(\left( cTb^{-r}+1\right)^n-1\right) \leq CT^nk^{2+\alpha-r}.
\end{align*}

\end{lemma}

\begin{proof}
We start in a similar fashion to the proof of proposition 2.7 by Bulinski and Fish in \cite{BF}, who initially adapt proposition 1 by Kelly and L\^e in \cite{KL}.


First, observe that for any $b \geq 1$, there are at most $b^{2m}$ values of $z \in \mathbb{T}^m$ that only have rational entries with reduced denominators being at most $b$. It follows that for any fixed $i$, there are at most $b^{2m}$ values of $j$ such that $x_i-x_j$ has only rational entries with denominators being at most $b$. Since there are $k$ choices for $i$, we find $S_b \leq kb^{2m}$.

Fix $Q > k^{1/(2m)}$ large enough such that $s_b = 0$ for all $b > Q$ and note that, by definition, $S_1 = 0$. By applying summation by parts, we can show the following,
\begin{align}
\sum_{b = 2}^\infty s_b \left(\left( cTb^{-r}+1\right)^n-1\right) &= \sum_{b = 2}^Q \left(S_b - S_{b-1}\right) \left(\left( cTb^{-r}+1\right)^n-1\right) \nonumber\\
&= S_Q\left(\left(cTQ^{-r}+1)\right)^n-1\right) \label{1}\\
&+ \sum_{2 \leq b < k^{1/(2m)}} S_b\left(\left(cTb^{-r}+1\right)^n - \left(cT\left(b+1\right)^{-r}+1\right)^n\right) \label{2}\\
&+ \sum_{k^{1/(2m)} \leq b < Q} S_b\left(\left(cTb^{-r}+1\right)^n - \left(cT\left(b+1\right)^{-r}+1\right)^n\right). \label{3}
\end{align}

We now seek to bound this expression when we send $Q \to \infty$, which only affects terms (\ref{1}) and (\ref{3}). For (\ref{1}), we recall that we always have $S_Q \leq k^2$, which implies this term converges to zero. For (\ref{3}), using the same inequality returns a telescopic sum, which can be evaluated in a relatively straight-forward manner. For (\ref{2}), we first consider the intermediate function $f(b) = \left(cTb^{-r}+1\right)^n$. By a binomial expansion and by noting $cT$ and $r$ are positive, $f(b)$ is convex for all $b > 0$, and hence by a linear approximation, $f(b+1) \geq f(b) + f'(b)$. Rearranging this and evaluating the derivative, we have the inequality,
\begin{align*}
\left(cTb^{-r}+1\right)^n - \left(cT\left(b+1\right)^{-r}+1\right)^n &\leq rcTnb^{-r-1}\left(cTb^{-r}+1\right)^{n-1}
\end{align*}

Using this expression, the bound $S_b \leq kb^{2m}$, and $2m-r-1 > 0$ to this expression, we find that there exists a constant $C_1$ depending on $c, r, m$ and $n$ such that, 
\begin{align*}
&\sum_{2 \leq b < k^{1/(2m)}} S_b\left(\left(cTb^{-r}+1\right)^n - \left(cT\left(b+1\right)^{-r}+1\right)^n\right)\\
&\leq \sum_{2 \leq b < k^{1/(2m)}} kb^{2m} rcTnb^{-r-1}\left(cTb^{-r}+1\right)^{n-1}\\
&\leq rcTnk\left(cT\cdot 2^{-r}+1\right)^{n-1}\sum_{2 \leq b < k^{1/(2m)}} b^{2m-r-1}\\
&\leq rcTnk\left(cT\cdot 2^{-r}+1\right)^{n-1} \int_{2}^{k^{1/(2m)} + 1} b^{2m-r-1}db\\
&\leq rcTnk\left(cT\cdot 2^{-r}+1\right)^{n-1}\left(\left(k^{1/(2m)}+1\right)^{2m-r} - 2^{2m-r}\right)\\
&\leq C_1 T^n k^{2-r/(2m)}
\end{align*}

Putting our bounds on (\ref{1}), (\ref{2}) and (\ref{3}) together, we find that by sending $Q \to \infty$,
\begin{align*}
\sum_{b = 2}^\infty s_b \left(\left( cTb^{-r}+1\right)^n-1\right) &\leq C_1T^n k^{2-r/(2m)} + k^2\left(\left(cT\lceil k^{1/(2m)}\rceil^{-r} + 1\right)^n - 1\right)\\
&\leq C T^n k^{2-r/(2m)}.
\end{align*}

For the case of $m = 1$, we can repeat the same arguments, but instead, for any $\alpha > 0$, we have $H_b \leq (kb)^{1+\alpha}$ for all $k > k_0(\alpha)$. This tighter bound on $H_b$ comes directly from proposition 1.3 of Alon and Peres \cite{AP}. Also note that $s_b = h_b$ when $m=1$, so $S_b \leq (kb)^{1+\alpha}$ for all $k > k_0(\alpha)$. By splitting the sum at the index $b = k$ instead of $b = k^{1/(2m)}$, we instead find the following corresponding bound,

\begin{align*}
\sum_{b = 2}^\infty s_b \left(\left( cTb^{-r}+1\right)^n-1\right) &\leq C_1T^n k^{2+\alpha-r} + k^2\left(\left(cT k^{-r} + 1\right)^n - 1\right)\\
&\leq C T^n k^{2+\alpha-r}.
\end{align*}

\end{proof}

We are now ready to prove theorem \ref{theorem1}. Unfortunately, there are a large amount of algebraic manipulations involved in the details of this proof, but this is done to make full use of $\mathcal{M}$ being indexed over many primes instead of just one, and minimising $k(\epsilon)$.

\begin{proof}[Proof of theorem \ref{theorem1}]
We use the Harmonic Analysis approach of Alon and Peres \cite{AP}. Let $\epsilon > 0$ and suppose $X = \{x_1, ..., x_k\}$ is a set of $k$ distinct points in $\mathbb{T}^m$ such that $MX$ is not $\epsilon$-dense for any $M \in \mathcal{M}$.

Let $h_\epsilon(z_1, z_2, ..., z_n) = \prod_{i=1}^n g_\epsilon(z_i)$ where the $g_\epsilon$ are bump functions from lemma \ref{lemma1}. It is straightforward to show using the integral formula for a Fourier coefficient that for all $m_i \in \mathbb{Z}$, $\widehat{h}_\epsilon(m_1, ..., m_n) = \prod_{i = 1}^n \widehat{g}_\epsilon(m_i)$. Applying lemma \ref{lemma1} to each $\widehat{g}_\epsilon$ will give us bounds on the Fourier coefficients of $h_\epsilon$.

For any $N > 1$, we can also index the elements of $\mathcal{M}$ by writing $M = M(\mathbf{p}) \in \mathcal{M}$ where $\mathbf{p} = (p_{1, 1}, p_{1, 2}, ..., p_{n, m}) \in P_N^{mn}$. It follows by the definition of $h_\epsilon$ that for every $M \in \mathcal{M}$, there exists a $\lambda_M \in \mathbb{T}^n$ such that,

\begin{align*}
0 &= \dfrac{1}{N^{mn}}\sum_{\mathbf{p} \in P_N^{mn}}\sum_{j = 1}^k h_\epsilon(Mx_j + \lambda_M)\\
&= \dfrac{1}{N^{mn}}\sum_{\mathbf{p} \in P_N^{mn}}\sum_{j = 1}^k\sum_{\mathbf{m} \in \mathbb{Z}^n}\widehat{h}_\epsilon(\mathbf{m})e_\mathbf{m}\left( Mx_j + \lambda_M\right).
\end{align*}

We note we attain equality in the Fourier series due to the Fourier coefficients converging to zero faster than any polynomial, and the periodicity of $h_\epsilon$. Here we let $e_\mathbf{m}(v) = e(2\pi i \mathbf{m} \cdot v)$ for $v \in \mathbb{T}^n$. Now we separate this into parts depending on the value of $\mathbf{m}$. Applying the triangle inequality, we can extract the $\mathbf{m} = \mathbf{0}$ term from the series. Let $S_T = \{\mathbf{m} \in \mathbb{Z}^n \mid 0 < \|\mathbf{m}\|_\infty \leq T\}$ be the lattice cube centred at 0, excluding 0. Then by rearranging our previous expression and applying the triangle inequality,
\begin{align}
k &\leq \left|\dfrac{1}{N^{mn}}\sum_{\mathbf{p} \in P_N^{mn}}\sum_{j = 1}^k\sum_{\mathbf{m} \in S_T} \widehat{h}_\epsilon(\mathbf{m})e_\mathbf{m}\left( Mx_j + \lambda_M\right)\right| \nonumber\\
&+ k \sum_{\| \mathbf{m} \|_{\infty} > T} \left| \widehat{h}_\epsilon(\mathbf{m})\right| \label{k inequality}
\end{align}

We now bound this second sum. By lemma \ref{lemma1} and the fact $\widehat{h_\epsilon}(\mathbf{m}) = \prod_{i = 1}^n \widehat{g_\epsilon}(m^i)$ (where the $i$ superscript denotes taking the $i$-th coordinate) is invariant under permutation of the coordinates, we know there exist constants $C$ and $C' = C'(n)$ such that,
\begin{align*}
\sum_{\| \mathbf{m} \|_{\infty} > T} \left| \widehat{h}_\epsilon(\mathbf{m})\right| &\leq 2n\sum_{\mathbf{m} \in \mathbb{Z}^n, m^1 \geq T}\left| \widehat{h}_\epsilon(\mathbf{m})\right|\\
&\leq 2n\sum_{\mathbf{m} \in \mathbb{Z}^n, m^1 \geq T} C \exp\left(-\sum_{i = 1}^n \sqrt{\epsilon m^i}\right)\\
&\leq C' \sum_{m^1 = T}^\infty \exp\left(-\sqrt{\epsilon m^1}\right).
\end{align*}

Directly from the proof of proposition 6.1 by Alon and Peres \cite{AP}, if we set $T = \lfloor 4/\epsilon \log^2(1/\epsilon) \rfloor$, then there exists some $\epsilon_1$ depending on $n$ such that when $\epsilon \in (0, \epsilon_1)$, this expression is less than $1/2$ . Using such a value of $\epsilon$, we can substitute this bound into (\ref{k inequality}) to find,
\begin{align*}
\dfrac{k}{2} &\leq \left|\dfrac{1}{N^{nm}} \sum_{\mathbf{p} \in P_N^{mn}}\sum_{j = 1}^k\sum_{\mathbf{m} \in S_T} \widehat{h}_\epsilon(\mathbf{m})e_\mathbf{m}\left( Mx_j + \lambda_M\right)\right|.
\end{align*}

Now squaring and using the Cauchy-Schwartz inequality twice, we can separate the Fourier coefficients from the exponential terms.
\begin{align*}
\dfrac{k^2}{4} &\leq \dfrac{1}{N^{2mn}} \left|\sum_{\mathbf{p} \in P_N^{mn}}\sum_{j = 1}^k\sum_{\mathbf{m} \in S_T} \widehat{h}_\epsilon(\mathbf{m})e_\mathbf{m}\left( Mx_j + \lambda_M\right)\right|^2\\
&\leq \dfrac{1}{N^{mn}} \sum_{\mathbf{p} \in P_N^{mn}}\left|\sum_{j = 1}^k\sum_{\mathbf{m} \in S_T} \widehat{h}_\epsilon(\mathbf{m})e_\mathbf{m}\left( Mx_j + \lambda_M\right)\right|^2\\
&\leq \dfrac{1}{N^{mn}}\sum_{\mathbf{p} \in P_N^{mn}}\left(\sum_{\mathbf{m} \in S_T}|\widehat{h}_\epsilon(\mathbf{m})|^2\right)\left(\sum_{\mathbf{m} \in S_T}\left|\sum_{j = 1}^k e_\mathbf{m}\left( Mx_j + \lambda_M\right)\right|^2 \right)\\
\end{align*}

By Bessel's inequality and the properties of the bump functions provided by lemma \ref{lemma1}, we find there is a constant $c_0$ such that,
\begin{align*}
\sum_{\mathbf{m} \in S_T}|\widehat{h}_\epsilon(\mathbf{m})|^2 &\leq \int_{\mathbb{T}^n} |h_\epsilon(z)|^2dz \\
&= \left(\int_\mathbb{T} g_\epsilon(z)^2 dz\right)^n = \dfrac{c_0}{\epsilon^n}.
\end{align*}

Using this and replacing constants by $c_1$, we do some careful manipulation of our products and sums.
\begin{align}
k^2 &\leq \dfrac{c_1}{N^{mn}\epsilon^n}\sum_{\mathbf{p} \in P_N^{mn}}\sum_{\mathbf{m} \in S_T}\left|\sum_{j = 1}^k e_\mathbf{m}\left( Mx_j + \lambda_M\right)\right|^2 \nonumber\\
&= \dfrac{c_1}{N^{mn}\epsilon^n}\sum_{\mathbf{p} \in P_N^{mn}}\sum_{\mathbf{m} \in S_T}\sum_{j = 1}^k\sum_{l = 1}^k e_\mathbf{m}\left(Mx_j - Mx_l\right) \nonumber\\
&= \dfrac{c_1}{\epsilon^n} \sum_{j = 1}^k \sum_{l = 1}^k \sum_{\mathbf{m} \in S_T} \dfrac{1}{N^{mn}}\sum_{\mathbf{p} \in P_N^{mn}} \prod_{u = 1}^n \prod_{v = 1}^m \exp(2\pi i m^u f_{u, v}(p_{u, v})(x_j^v-x_l^v)) \nonumber\\
&= \dfrac{c_1}{\epsilon^n}  \sum_{j = 1}^k \sum_{l = 1}^k  \left[\sum_{\|\mathbf{m}\|_\infty \leq T}\left(\prod_{u = 1}^n \prod_{v = 1}^m \sum_{p \in P_N}\dfrac{1}{N} \exp(2\pi i m^u f_{u, v}(p)(x_j^v-x_l^v))\right) - 1\right] \nonumber\\
&=: \dfrac{c_1}{\epsilon^n}  \sum_{j = 1}^k \sum_{l = 1}^k I_{j, l}. \label{sumI}
\end{align}

Here we have set $I_{j, l}$ to be the term inside the square brackets, which we isolate from the outer sum to simplify some of the working. Further simplifying this term,
\begin{align}
I_{j, l} &= \sum_{\|\mathbf{m}\|_\infty \leq T}\left(\prod_{u = 1}^n \prod_{v = 1}^m \sum_{p \in P_N}\dfrac{1}{N} \exp(2\pi i m^u f_{u, v}(p)(x_j^v-x_l^v))\right) - 1 \nonumber\\
&= \prod_{u = 1}^n\left( \sum_{w = -T}^T \prod_{v = 1}^m \sum_{p \in P_N} \dfrac{1}{N}\exp(2\pi i w f_{u, v}(p)(x_j^v-x_l^v))\right) - 1 \nonumber\\
&\leq \left( \sum_{w = -T}^T \prod_{v = 1}^m \left|\sum_{p \in P_N} \dfrac{1}{N}\exp(2\pi i w f_{\cdot, v}(p)(x_j^v-x_l^v))\right|\right)^n - 1 \label{Ijl}
\end{align}

The inequality comes from the fact that $I_{j, l}$ is real as, by matching the positive and negative $w$ terms together, we have the sum of complex numbers and their conjugates. Furthermore, we introduce $f_{\cdot, v}$ to represent the choice of $u$ which maximises the outer product. We now take $N \to \infty$, and as the outer terms have no dependence on $N$, it is enough for us to study the inner sum of (\ref{Ijl}), on which we apply our lemmas.

Suppose for some $j, l$, we have $x_j-x_l$ containing at least one irrational entry and, without loss of generality, suppose the first entry $x_j^1 - x_l^1$ is irrational. Then, when $w$ is non-zero, we may apply lemma \ref{lemma2} to observe as $N \to \infty$, the $v = 1$ term of the product over $v$ in (\ref{Ijl}) converges to zero, thus the product converges to zero. If $w = 0$, the inner sum is identically one. It follows $I_{j, l}$ can be bounded from above by an expression which approaches 0 as $N \to \infty$ in this case.

Next, suppose $x_j-x_l$ only has rational entries and is non-zero, and suppose that the entry with the largest denominator when written in reduced form is $a/b$. Without loss of generality, we may assume $x_j^1 - x_l^1 = a/b$. By applying lemma \ref{lemma3} we observe that for any choice of $\epsilon' > 0$, there is a constant $c_2$ such that,

\begin{align*}
&\lim_{N \to \infty}\sum_{w = -T}^T \prod_{v = 1}^m  \Big|\sum_{p \in P_N} \dfrac{1}{N}\exp(2\pi i w f_{\cdot, v}(p)(x_j^v-x_l^v))\Big|\\
&\leq \sum_{w = -T}^T \Big|\lim_{N \to \infty}\sum_{p \in P_N} \dfrac{1}{N}\exp(2\pi i w f_{\cdot, 1}(p)(x_j^1-x_l^1))\Big| \leq c_2Tb^{-1/L + \epsilon'} + 1
\end{align*}
It naturally follows that as $N \to \infty$, $I_{j,l}$ can be bounded by $(c_2Tb^{-1/L + \epsilon'} + 1)^n - 1$ when at least one of the entries of $x_j - x_l$ is the reduced non-zero rational $a/b$. Finally, we note that when $x_j - x_l = 0$, by the uniqueness of our points, we must have $j = l$ which can only occur $k$ times. By counting the number of points in the lattice cube $\|\mathbf{m}\|_\infty \leq T$, we must have $I_{j, j} = (2T+1)^n - 1$ for $j = 1, ..., k$.

Grouping our summation over $j$ and $l$ by the bounds we have on $I_{j, l}$, and by counting the number of times we fall into each of our three cases, we can form a bound for $k^2$ using (\ref{sumI}). Recall that the number of times we fall into the rational case is exactly $s_b$ as defined previously. Substituting our bounds,
\begin{align*}
k^2 &\leq \dfrac{c_1}{\epsilon^n}\left(k \left(\left(2T + 1\right)^n - 1\right) + 2\sum_{b = 2}^\infty s_b\left(\left(c_2T b^{-1/L + \epsilon'} + 1\right)^n - 1\right) \right).
\end{align*}

The sum over $b$ can be bounded by the use of lemma \ref{lemma4}, with $r = 1/L-\epsilon'$ which is in the interval $(0, 1)$ for a small enough choice of $\epsilon'$. Replacing constants by $c_i$ as they appear, each of which is independent of $T$ and $k$,
\begin{align*}
k^2 &\leq \dfrac{c_1}{\epsilon^n}\left(k\left(\left(2T+1\right)^n - 1\right) + c_3 T^n k^{2-(1/L-\epsilon')/(2m)}\right)\\
\Longrightarrow k^2 &\leq c_4 \dfrac{T^n}{\epsilon^n} k^{2-(1/L-\epsilon')/(2m)}\\
\Longrightarrow k &\leq c_5 \left(\dfrac{T}{\epsilon}\right)^{2Lmn + 2\epsilon' nL^2m/(1-\epsilon'L)}
\end{align*}

By recalling our choice of $T$, we note that for any choice of $\delta > 0$, there exists an $\epsilon$ small enough such that $T/\epsilon \leq 1/\epsilon^{2+\delta}$. By incorporating our choice of $\epsilon'$ and the constant $c_5$ into $\delta$, it follows that for any choice of $\delta > 0$, there exists a sufficiently small $\epsilon_0$ such that if $\epsilon \in (0, \epsilon_0)$ and $k > 1/\epsilon^{2Lm(m+1)n + \delta}$, then there must exist a $M \in \mathcal{M}$ such that $MX$ is $\epsilon$-dense in $\mathbb{T}^n$, and so we attain the quantitative Glasner property.

If instead $m = 1$, we can apply the tighter bound of lemma \ref{lemma4} with the same value of $r = 1/L - \epsilon'$. Then for any choice of $\alpha > 0$, if $k \geq k_0(\alpha)$,
\begin{align*}
k^2 &\leq \dfrac{c_1}{\epsilon^n}\left(k\left(\left(2T+1\right)^n-1\right) + c_3 T^n k^{2+\alpha-1/L+\epsilon'}\right)\\
\Longrightarrow k &\leq c_4 \left(\dfrac{T}{\epsilon}\right)^{nL + (\epsilon' + \alpha')nL^2/(1-\epsilon'L-\alpha L)}
\end{align*}

We know that for any choice of $\delta > 0$, $T/\epsilon \leq 1/\epsilon^{2+\delta}$. By incorporating $\alpha$ and $\epsilon'$ into this choice, in addition to choosing our upper bound for $\epsilon$, $\epsilon_0$ as depending on $\alpha$, we find that for any choice of $\delta > 0$, there exists a sufficiently small $\epsilon_0$ such that if $\epsilon \in (0, \epsilon_0)$ and $k > 1/\epsilon^{2Ln + \delta}$, then there exists a $M \in \mathcal{M}$ such that $MX$ is $\epsilon$-dense in $\mathbb{T}^n$.

\end{proof}

\section{Proof of theorem \ref{theorem2}}

We now restrict theorem \ref{theorem1} to the case of each entry in our set of matrices depending on a single choice of prime, rather than independent choices of primes for each polynomial.

There are a number of definitions and statements from \cite{BF} which will be helpful here. We first define \textit{multiplicative complexity}, a way of bounding how linear combinations of polynomials may introduce common factors in their coefficients.

\begin{definition}
The vector $P(x) = [P_1(x), ..., P_r(x)] \in (\mathbb{Z}[x])^r$ has \textit{multiplicative complexity} $Q$ if, for all $a = (a_1, ..., a_r) \in \mathbb{Z}^r$ and $q \in \mathbb{Z}$ with $\gcd(a_1, ..., a_r, q) = 1$, we have the polynomial
\begin{align*}
(P(x) - P(0)) \cdot a = \sum_{i = 1}^L b_jx^j
\end{align*}
satisfying $\gcd(b_1, ..., b_L, q) \leq Q$.
\end{definition}

We will also require Corollary 2.4 from \cite{BF}, however we observe from the proof it can easily be extended to the case of non-square matrices.

\begin{lemma}
\label{lemma5}
Let $M(x) \in M_{n \times m}(\mathbb{Z}[x])$ be a matrix with polynomial entries and $\mathbf{m} \in \mathbb{Z}^n \setminus \{0\}$ such that the entries of $\mathbf{m}^t(M(x) - M(0))$ are linearly independent over $\mathbb{Z}$. Then $\mathbf{m}^t M(x)$ has multiplicative complexity $Q$, where
\begin{align*}
Q = Q(M(x), \mathbf{m}) = m! (n\|M(x) - M(0)\| \|\mathbf{m}\|_\infty)^m.
\end{align*}
\end{lemma}
Here $\|M(x) - M(0)\|$ denotes the largest polynomial coefficient in any entry of $M(x) - M(0)$.
\begin{proof}
First, we make note of proposition 2.2 of Bulinski and Fish, which states that for any linearly independent vectors $v_1, ..., v_m \in \mathbb{Z}^{r}$ and coprime integers $a_1, ..., a_m, q$, we have $\gcd\left(\sum_{i=1}^m a_iv_i, q\right) \leq m! \max_{i}\| v_i\|_\infty^m$, where the $\gcd$ is computed by taking the $\gcd$ of coordinate entries of the first argument and $q$ \cite{BF}. 

Let $P(x) = [p_1(x), ..., p_m(x)] = \mathbf{m}^t\left(M(x) - M(0)\right)$, and note that $P(0) = 0$. It is given that these polynomials are linearly independent over $\mathbb{Z}$, so by treating them as vectors we can apply the proposition. It follows that for any coprime integers $a_1, ..., a_m, q \in \mathbb{Z}$, we also have $\gcd(a_1p_1(x) + ... + a_mp_m(x), q) \leq m! \max_i \| p_i\|_\infty^m$, where the $\gcd$ is computed by taking the $\gcd$ of the coefficients of the first argument and the second argument. 

This is exactly the definition of multiplicative complexity, where $a = (a_1, ..., a_m)$. In our treatment of polynomials as vectors, $\|p_i\|_\infty$ refers to the largest coefficient on $p_i$, which must be bounded by $n\|\mathbf{m}\|_\infty \|M(x) - M(0)\|_\infty$. It follows $\mathbf{m}^tM(x)$ has the desired multiplicative complexity.
\end{proof}

Finally, we further require a version of lemma \ref{lemma3} from Shparlinski on exponential sums for polynomials \cite{Sh}.

\begin{lemma}
\label{lemma6}
Let $P(x) = a_0 + a_1x + ... + a_Lx^L \in \mathbb{Z}[x]$ and let $b \in \mathbb{Z}$ be such that $\gcd(a_0, ..., a_L, b) = 1$. Then, for any choice of $\delta > 0$, there exists a constant $C$ independent of $b$ such that,
\begin{align*}
\sum_{1 \leq r \leq b, \gcd(r, b) = 1} \exp\left(2\pi i \dfrac{1}{b}P(r)\right) &\leq Cb^{1 - 1/L + \delta}.
\end{align*}
\end{lemma}

\begin{proof}[Proof of theorem \ref{theorem2}]
Much of this proof is similar to that of theorem \ref{theorem1} and theorem \ref{theoremBF} by Bulinski and Fish, so for the sake of brevity we will place emphasis on where this proof is different and outline other steps.

Begin by letting $\epsilon > 0$, and $X = \{x_1, ..., x_k\}$ such that for all primes $p$, the set $M(p)X$ is not $\epsilon$-dense in $\mathbb{T}^n$. Following the same steps as theorem \ref{theorem1}, we continue until we have the corresponding inequality,
\begin{align*}
k^2 &\leq \dfrac{c_1}{\epsilon^n} \sum_{\mathbf{m} \in S_T} \sum_{1 \leq  j, l \leq k} \dfrac{1}{N}\sum_{p \in P_N} e_\mathbf{m}\big(M(p)(x_j - x_l)\big)
\end{align*}

We now take $N \to \infty$, and consider different cases depending on the value of $x_j - x_l$. First, suppose that $x_j - x_l$ has at least one irrational entry. Using the linear independence of the $f_{i, j}$, set $P(x) = [P_1(x), ..., P_m(x)] :=\mathbf{m}^t(M(x) - M(0))$ to be a vector of linearly independent polynomials over $\mathbb{Z}$ and hence also over $\mathbb{Q}$. As in the proof of theorem 2.8 of \cite{BF}, $P(x) \cdot (x_j - x_l)$ will have at least one irrational non-constant coefficient. It immediately follows from lemma \ref{lemma2} that as $N \to \infty$, the inner sum will converge to 0 and the term will not contribute to anything in the overall expression.

Now suppose $x_j - x_l$ is rational and non-zero, and suppose that $b$ is the smallest positive integer such that $b(x_j - x_l) \in \mathbb{Z}^m$. We can then write $x_j - x_l = \dfrac{1}{b}a$ where $a \in \mathbb{Z}^m$ and $\gcd(a^1, ..., a^m, b) = 1$. Recall that by lemma \ref{lemma5}, the multiplicative complexity of $\mathbf{m}^t (M(x) - M(0))$ is given by $Q$, where, by considering that $\mathbf{m} \in S_T$,
\begin{align*}
Q &\leq m!(n \| M(x) - M(0)\| \| \mathbf{m} \|_\infty)^m\\
&\leq m! (n \| M(x) - M(0)\| )^m T^m.
\end{align*}

Let $P(x) = \mathbf{m}^t M(x) a$. Then, due to Dirichlet's equidistribution of polynomials evaluated at the primes and lemma \ref{lemma6}, we observe that for any choice of $\delta > 0$, there exists a constant $C_{L, \delta}$ such that
\begin{align*}
\lim_{N \to \infty} \dfrac{1}{N} \Bigg|\sum_{p \in P_N} e_{\mathbf{m}}\left(\dfrac{1}{b}M(p)a\right)\Bigg| &= \lim_{N \to \infty} \dfrac{1}{N} \Bigg|\sum_{p \in P_N} \exp\left(2\pi i \dfrac{1}{b}P(p)\right)\Bigg|\\
&= \dfrac{1}{\phi(b)}\Bigg|\sum_{1 \leq r \leq b, \gcd(r, b) = 1} \exp\left(2\pi i \dfrac{1}{b}P(r)\right)\Bigg|\\
&\leq C_{L, \delta} \left(\dfrac{Q}{b}\right)^{1/L - \delta},
\end{align*}
where in the last inequality, we have used that $\phi(b) \gg b^{1-\delta}$ for any choice of $\delta > 0$ where $\phi$ is Euler's totient function. Finally, in the case that $x_j-x_l$ is zero, we must have $j = l$ and the inner sum over $P_N$ will be identically 1. This case occurs $k$ times exactly.

Putting these cases together, we note that the number of times we fall into the rational case is exactly $h_b$ by definition. It follows that,
\begin{align*}
k^2 \leq \dfrac{c_1}{\epsilon^n} \sum_{\mathbf{m} \in S_T} \left(\sum_{b = 2}^\infty h_b \left(\dfrac{Q}{b}\right)^{1/L - \delta} + k\right).
\end{align*}

This is nearly the same bound obtained by Bulinski and Fish in theorem 2.8 \cite{BF}, with the only difference being the $h_b$ counting points in $\mathbb{T}^m$ instead of $\mathbb{T}^n$. Applying proposition 2.7 of Bulinski and Fish, we get the bound on the sum,
\begin{align*}
\sum_{b=2}^\infty h_b b^{-1/L+\delta} \leq c_2 k^{2-(1/L-\delta)/(m+1)}
\end{align*}

Furthermore, in a similar manner to lemma \ref{lemma4}, it is possible to show that when $m = 1$, by using the tighter bounds available on $H_b$, we will find,
\begin{align*}
\sum_{b=2}^\infty h_b b^{-1/L+\delta} \leq c_2 k^{2-1/L+\delta}.
\end{align*}

Continuing the proof of theorem \ref{theoremBF} by Bulinski and Fish nearly verbatim with these bounds in mind gives the final result.
\end{proof}

\section{Acknowledgements}

I would like to thank Professor Alexander Fish for acting as a supervisor and reviewer during the construction of this paper, without whom this project would have not been possible. Additionally, I would like to thank the anonymous reviewer for their invaluable feedback in providing corrections and making this paper more digestible.

\bibliography{bibliography.bib}

\end{document}